\documentclass[12pt]{amsart}
\usepackage[utf8]{inputenc}
\usepackage[T1]{fontenc}
\DeclareFontFamily{U}{mathx}{}
\DeclareFontShape{U}{mathx}{m}{n}{<-> mathx10}{}
\DeclareSymbolFont{mathx}{U}{mathx}{m}{n}
\DeclareMathAccent{\widehat}{0}{mathx}{"70}
\DeclareMathAccent{\widecheck}{0}{mathx}{"71}
\usepackage[usenames]{color}
\usepackage[usenames]{color}
\usepackage{graphicx} 
\usepackage{mathtools,amssymb}
\usepackage{amsthm,amscd}
\usepackage{amsmath,pdfsync,verbatim,graphicx,epstopdf,enumerate,xcolor}
\usepackage[normalem]{ulem}
\usepackage[colorlinks=true]{hyperref}
\usepackage{mathrsfs}
\mathtoolsset{showonlyrefs} 
\hypersetup{allcolors=blue}

\usepackage{cite}
\usepackage{xcolor}
\usepackage[utf8]{inputenc}
\usepackage{tikz}
\usetikzlibrary{calc,trees,positioning,arrows,chains,shapes.geometric,decorations.pathreplacing,decorations.pathmorphing,shapes,matrix,shapes.symbols}

\numberwithin{equation}{section}
\newcommand{\D}{\mathrm{d}}

\renewcommand{\O}{\Omega}

\newcommand{\lb}{\left (}
\newcommand{\rb}{\right)}

\newcommand{\PD}{\partial}

\newcommand{\R}{\mathbb{R}}

\newcommand{\wt}{\widetilde}

\allowdisplaybreaks

\newtheorem{theorem}{Theorem}[section]

\newtheorem{corollary}{Corollary}[section]

\newtheorem{lemma}{Lemma}[section]

\newtheorem{proposition}{Proposition}[section]
\newtheorem{remark}{Remark}[section]

\newcommand{\Nc}{\mathcal{N}}
\newcommand{\vp}{\varphi}
\newcommand{\Oc}{\mathcal{O}}

\newcommand{\ve}{\varepsilon}

\newcommand{\Bc}{\mathcal{B}}

\newcommand{\Sb}{\mathbb{S}}

\renewcommand{\O}{\Omega}

\renewcommand{\o}{\omega}
\renewcommand{\wt}{\widetilde}

\usepackage{bm}
\newcommand{\commutator}[2]{[#1,#2]}

\title[Direct and Inverse Problem for bi-wave operator]{Direct and Inverse Problem for bi-wave equation with time-dependent coefficients from partial data}

\author[Bhattacharyya and Kumar]{Sombuddha Bhattacharyya and Pranav Kumar}
\address{Department of Mathematics, Indian Institute of Science Education and Research, Bhopal.
\newline
E-mail:{\tt \ sombuddha@iiserb.ac.in}}
\address {Yau Mathematical Sciences Center, Tsinghua University, Beijing, China.
\newline
E-mail:{\tt \ kumarp@tsinghua.edu.cn, pk958675maths@gmail.com}}

\begin{document}

\begin{abstract}
In this article, we study a direct and an inverse problem for the bi-wave operator $(\Box^2)$ along with second and lower order time-dependent perturbations. In the direct problem, we prove that the operator is well-posed, given initial and boundary data in suitable function spaces. In the inverse problem, we prove uniqueness of the lower order time-dependent perturbations from the partial input-output operator. The restriction in the measurements are considered by restricting some of the Neumann data over a portion of the lateral boundary. 
\end{abstract}


\maketitle

\section{Introduction and Statement of Results}
In this article, we study a perturbed bi-wave operator, defined on a bounded space-time domain. Bi-wave operators appear in the study of simplified Ginzburg-Landau type model for $d$-wave superconductors \cite{PhysRevB.53.12481,Qiang-d-wave,Xiaobing-Michael-1}. Bi-wave operators are also useful to describe systems consisting of two waves. A long standing problem with bi-wave is described in \cite[Chapter 8.6]{Isakov_book}, where the principal operator is a composition of two wave operators with different speeds. We notice that, even with the same wave speeds, recovering lower order perturbations of a bi-wave operator requires investigation. We believe that our results here can be seen as a step towards dealing with more general bi-wave type operators. In this article we demonstrate a direct problem and an inverse problem for the perturbed bi-wave operator from initial and boundary conditions. The direct problem proves the well-posedness of the operator given the boundary and initial data on suitable function spaces. In the inverse problem, we prove that the coefficients of second and lower order perturbations are uniquely determined by the input-output operator measured at the boundary of the space-time domain.

Here we describe the basic setup of the operator under consideration.
Let $\Omega\subset \R^n, n\geq 3$, be a bounded simply connected open set with smooth boundary $\partial\Omega$. For $T>0$, let us denote $Q:=(0,T)\times \Omega$ and its lateral boundary $\Sigma:=(0,T)\times \partial\Omega$. We consider the following perturbed bi-wave operator 
\begin{equation}\label{operator}
\mathscr{L}_{A,B,C,q}:=\Box^{2}+A(t,x)\Box+B(t,x)\partial_{t} +C(t,x)\cdot \nabla_{x}+q(t,x) \quad \text { in } Q,
\end{equation}
where $\Box^{2}:= \Box \circ \Box =(\partial_{t}^{2}-\Delta_{x})^{2}$ is the bi-wave operator, and $A\in W^{3,\infty}(Q)$, $B\in W^{2,\infty}(Q)$, $q\in L^{\infty}(Q)$ and $C=(C_1,\cdots, C_n)\in \left(W^{2,\infty}(Q)\right)^n$ are the perturbation coefficients. We assume that all the coefficients are time-dependent and real-valued. 
We consider the following initial boundary value problem (IBVP)
\begin{equation}\tag{IBVP}
\begin{aligned}\label{IBVP}
 \begin{cases}
 \mathscr{L}_{A,B,C,q}u =0 & \text { in } Q,\\
 \left (u|_{t=0}, \partial_{t}u|_{t=0}, \partial_{t}^{2}u|_{t=0}, \partial_{t}^{3}u|_{t=0}\right)=\left(\psi_{0},\psi_{1},\psi_{2},\psi_{3}\right) &\text { on }\Omega,\\
 \left( u, \Box u\right)= \left ( f, g\right) &\text { on } \Sigma,
 \end{cases}
 \end{aligned}
 \end{equation}
where $\psi_0, \psi_1, \psi_2, \psi_3$ are the initial and $f, g$ are the boundary data.
Here we work with the time-dependent Sobolev spaces $C^k\big([0,T]; X\big)$, which denote the spaces of $k$-times continuously differentiable functions from $[0,T]$ to a Banach space $X$.

In the direct problem, we prove that \eqref{IBVP} is well-posed, given that the perturbation coefficients and the initial and boundary data are regular enough and satisfy natural compatibility conditions.
\begin{theorem}\label{IBVP-Biwave-wellposedness}
Let $A\in W^{3,\infty}(Q)$, $B, C \in W^{2,\infty}(Q)$ and $q\in L^{\infty}(Q)$ be time-dependent coefficients. Assume $\psi_{j} \in H^{4-j}(\Omega)$ for $j=0,1,2,3$, $f\in H^{4}(\Sigma)$ and $g\in H^{2}(\Sigma)$ satisfying the compatibility conditions
\[
(\psi_{0},\psi_{2}-\Delta_{x}\psi_{0})|_{\partial\Omega}= (f,g)|_{(\{t=0\}\times \partial\Omega)}
\quad\mbox{and}\quad (\psi_{1},\psi_{3}-\Delta_{x}\psi_{1})|_{\partial\Omega}= \partial_{t}(f,g)|_{(\{t=0\}\times \partial\Omega)}.    
\]
Then there exists a unique solution $u(t)$ of \eqref{IBVP} such that
\[  u(t), \Box u(t) \in C\lb [0,T];H^{2}(\Omega)\rb\cap C^1\lb [0,T];H^{1}(\Omega)\rb\cap C^2\lb [0,T];L^{2}(\Omega)\rb 
\]
and a constant $C>0$ such that
\begin{equation}\label{Energy_Est-1}
\begin{aligned}
\max_{0\leq t\leq T}&\left(\sum_{j,k=0,1,2} \lVert \partial^j_{t}(\Box^k u)(t) \rVert_{H^{2-k}(\Omega)}\right) + \sum_{k=0,1} \lVert\partial_{\nu}\Box^k u\rVert_{H^1(\Sigma)}\\
&\leq C \Bigl\{ 
\sum_{k=0,1} \|\psi_{k}\|_{H^{4-k}(\Omega)} 
+ \sum_{k=0,1} \|(\psi_{2+k}-\Delta_{x}\psi_{k})\|_{H^{1-k}(\Omega)}
+ \lVert f\rVert_{H^{4}(\Sigma)}+\lVert g\rVert_{H^{2}(\Sigma)}
\Bigl\}.
\end{aligned}
\end{equation}
\end{theorem}

In the inverse problem, we prove unique determination of time-dependent coefficients $A$, $B$, $C$, and $q$ from the data observed on a portion of the boundary $\partial Q$. To formulate the problem, we begin by introducing the partial input-output operator. For a fixed $\omega_{0}\in\mathbb{S}^{n-1}$, 
we define $\omega_{0}$-shadowed and $\omega_{0}$-illuminated faces of $\partial\Omega$, respectively, as follows (see \cite{Bukhgeim-Uhlmann-2002})
\begin{equation*}
\partial\Omega_{+,\omega_{0}}:=\left\{x\in\partial\Omega:\ \nu(x)\cdot\omega_{0}\geq 0 \right\}\
\mbox { and }\,
\partial\Omega_{-,\omega_{0}}:=\left\{x\in\partial\Omega:\ \nu(x)\cdot\omega_{0}< 0 \right\},
\end{equation*}
where $\nu(x)$ is the outward unit normal vector to the boundary $\partial\Omega$ at $x$.
Corresponding to $\partial\Omega_{\pm}, \omega_{0}$ we denote portions of the lateral boundary of $Q$ as
\begin{equation}\label{lateral_bddy_part}
     \Sigma_{\pm,\omega_{0}}:=(0,T)\times\partial\Omega_{\pm, \omega_{0}}.
\end{equation}
Here we denote $G=(0,T)\times G^{\prime}$ where $G^{\prime}$ is an open neighborhood of $\partial\Omega_{-,\omega_{0}}$ in $\partial\Omega$. We define 
\[ 
\begin{aligned}
   \mathcal{E}(\Sigma, \Omega) := H^{4}(\Sigma) \times H^{2}(\Sigma) \times \prod_{j=0}^{3} H^{4-j}(\Omega)
\quad 
   \mbox{and} \quad \mathcal{K}(G, \Omega):=&(L^{2}(G))^{2}\times (L^{2}(\Omega))^{3},
\end{aligned}
\]
the input and output data spaces, respectively. Corresponding to the problem \eqref{IBVP}, we define the partial input-output operator
\[
\Nc_{A,B,C,q}: \mathcal{E}(\Sigma, \Omega)\to \mathcal{K}(G, \Omega), 
\]
\begin{equation}\label{Measurement}
\mathcal{F}=(f,g,\psi_{0}, \psi_{1},\psi_{2},\psi_{3})\mapsto \Nc_{A,B,C,q}(\mathcal{F})=(\partial_{\nu}u|_{G}, \partial_{\nu}(\Box u)|_{G}, u|_{t=T}, \partial_{t}u|_{t=T}, \partial_{t}^2u|_{t=T}).
\end{equation} 
In the next theorem, we state that the partial input-output operator $\mathcal{N}_{A, B, C, q}$ uniquely determines the time-dependent lower order perturbations $A$, $B$, $C$, and $q$.
\begin{theorem}\label{Main-Result}
Let $Q=(0, T)\times \Omega$, where $T>0$ and $\Omega\subset \R^n$, $n\geq 3$ be a bounded simply connected open set with smooth boundary. Let $A^{(i)}\in W^{3,\infty}(Q)$, $B^{(i)}, C^{(i)}\in W^{2,\infty}(Q)$ and $q^{(i)}\in L^{\infty}(Q)$, for $i=1,2$. Let $u_i$, $i=1,2$ be solutions to \eqref{IBVP} when $(A,B,C,q)=(A^{(i)},B^{(i)},C^{(i)},q^{(i)})$. If $\partial_{\nu}^k A^{(1)}=\partial_{\nu}^k A^{(2)}$, for $k=0,1$; $B^{(1)}=B^{(2)}$ and $C^{(1)}=C^{(2)}$ on $\Sigma$, then  
\begin{equation}\label{IO-operator}
\Nc_{A^{(1)},B^{(1)},C^{(1)},q^{(1)}}(\mathcal{F})= \Nc_{A^{(2)},B^{(2)},C^{(2)},q^{(2)}}(\mathcal{F}), \quad \text { for all } \mathcal{F}\in \mathcal{E}(\Sigma, \Omega),
\end{equation}
implies that $A^{(1)}=A^{(2)},\ B^{(1)}=B^{(2)}$, $C^{(1)}=C^{(2)}$ and $q^{(1)}=q^{(2)}$ in $Q$.  
\end{theorem}

The study of inverse problems for hyperbolic partial differential equations (PDEs) has attracted significant attention in recent years. Partly due to the wide range of applications and partly because of its rich theoretical literature. For the direct problem, we use the method of Galerkin approximation, as developed in \cite{Lions_Magenese_Book,Lassas_Book} modified suitably for a system of hyperbolic operators. Some of the ideas here follow from \cite{RM-Matrix-potential}. The proof of the inverse problem relies on obtaining appropriate integral identities for the differences of the perturbation coefficients using the input-output operators. The integral identities we obtain, are known as the Light Ray Transforms (LRTs), which have been proved to be injective in \cite{Stefanov_1989,Plamen_Yang_2018,Siamak_2018,Ali_Joonas_lauri_LRT_2021,Krishnan_Senapati_Vashisth_2020}. To obtain the LRTs, we construct a sufficiently large class of solutions, known as the Geometric Optics (GO) solutions. In particular, we construct growing and decaying type GO solutions for the operator $\mathcal{L}_{A,B,C,q}$ and its adjoint. Motivated by the existing literature on the hyperbolic inverse problems,  we construct the GO solutions using an interior Carleman weight. Furthermore, we prove a boundary Carleman estimate to control solutions on the part of the boundary where the output data is not measured. To the best of our knowledge, this is the first result concerning the coefficient determination problems for higher-order wave operators available in the literature.

The pioneering works of Sylvester and Uhlmann~\cite{Global_uniqueness}, and Rakesh and Symes~\cite{Rakesh-Symes98} motivated a lot of researchers to investigate inverse problems for recovering unknown perturbations of PDEs. In \cite{Isakov-damping91}, the unique recovery of the time-independent damping coefficient and potential was achieved using the same boundary data. Afterward, extensive progress has been made in the literature on recovering time-dependent coefficients from boundary measurements of solutions to initial boundary value problem for second-order hyperbolic operators; see, for instance, \cite{Ramm-Sjostrand91,Isakov_1991,Eskin2007,Kian_Damping_2016,Kian-17,Kian_Lauri_2019,K.Vashisth,Ali-Joonas-kian-Lauri-21,Real_Principal_Op,Liu_Saksala_Lili} and the references therein. There have also been extensive results concerning stability issues, primarily based on the construction of GO solutions; we refer the reader to see \cite{Stefanov_UHl_2005,Bell_Jell_Yama_2006,Bellassoued_Dos_2011,Ibtissem_Stability_2015,Kian-Stability-potential,Soumen_Stability} and the references therein.

The rest of the article is organized as follows. In Section~\ref{Direct-Problem}, we discuss the well-posedness of the ~\eqref{IBVP}. Section \ref{GO-Soln}, we construct special classes of GO solutions to $\mathscr{L}_{A, B, C,q}u=0$ based on the Carleman estimate. Finally, Section~\ref{Proof_main_thm} gives the proof of Theorem~\ref{Main-Result}.

\section{The Direct Problem}\label{Direct-Problem}
In this section, we discuss the existence and uniqueness of \eqref{IBVP}. We start by introducing some of the notations that will be useful throughout the section. For $s \in \mathbb{R}$, we denote $\mathbf{H}^s(X)$ as the space of vector-valued functions defined on $X$ with components in the Sobolev space $H^s(X)$. A similar notation will also be used for other vector-valued function spaces, such as $\mathbf{L}^2(X)$ and $\mathbf{C}^k(X)$, where $X = \Omega$ or $\partial\Omega$. For $\vec{u} = \left(u_1, u_2\right)^T \in \mathbf{H}^1(Q)$, we define
\[
 \lVert \vec{u} \rVert_{\mathbf{L}^2(Q)}^2 := \sum_{j=1}^{2} \int_Q |u_j(t, x)|^2 \, dx \, dt \quad 
	\mbox{and} \quad \nabla_x \vec{u} := \left( \nabla_x u_1, \nabla_x u_2 \right)^T.
\]
We prove Theorem \ref{IBVP-Biwave-wellposedness} by converting \eqref{IBVP} to a system of wave equations. By a suitable modification of the existing well-posedness results for systems of wave operators \cite{Lions_Magenese_Book,RM-Matrix-potential}, we finally conclude the proof.
Let us consider the system of wave equations
\begin{equation}\label{IBVP for first order vector}
\begin{aligned}
\begin{cases}
\left(\Box +\mathbf{P}_{\mathbf{B},\mathbf{C},\mathbf{q}}\right)\vec{u} := \left(\Box + \mathbf{B}(t,x,\partial_{t}) + \mathbf{C}(t,x,\nabla_{x}) + \mathbf{q}(t,x)\right)\vec{u} = \vec{0} & \text{ in } Q,\\
\vec{u}(0,x) = \vec{\Phi}(x), \quad \quad \partial_{t}\vec{u}(0,x) = \vec{\Psi}(x) & \text{ in } \Omega,\\
\vec{u}(t,x) = \vec{f}(t,x) & \text{ on } \Sigma,
	 \end{cases}
	 \end{aligned} 
  \end{equation}
\[\begin{gathered}
\mbox{where} \quad \vec{u}(t,x) = \begin{pmatrix}
    u_1(t,x)\\
    u_2(t,x)
\end{pmatrix}, \quad \vec{f}(t,x) = \begin{pmatrix}
    f_1(t,x)\\
    f_2(t,x)
\end{pmatrix},\\
\vec{\Phi}(x) = \begin{pmatrix}
    \Phi_1(x)\\
    \Phi_2(x)
\end{pmatrix},\quad 
\vec{\Psi}(x) = \begin{pmatrix}
    \Psi_1(x)\\
    \Psi_2(x)
\end{pmatrix},\quad 
\mathbf{q}(t,x) := \begin{bmatrix}
q_{11}(t,x) & q_{12}(t,x) \\
q_{21}(t,x) & q_{22}(t,x)
\end{bmatrix},\\ 
\mathbf{B}(t,x,\partial_{t}) := 
\begin{bmatrix}
    b_{11}(t,x)\partial_{t} & b_{12}(t,x)\partial_{t} \\
    b_{21}(t,x)\partial_{t} & b_{22}(t,x)\partial_{t} 
\end{bmatrix},
\quad 
\mathbf{C}(t,x,\nabla_{x}) := 
\begin{bmatrix}
    c_{11}(t,x)\cdot \nabla_{x} & c_{12}(t,x)\cdot \nabla_{x} \\
    c_{21}(t,x)\cdot \nabla_{x} & c_{22}(t,x)\cdot \nabla_{x}
\end{bmatrix},
\end{gathered}
\]
with $ b_{ij}(t,x) \in W^{2,\infty}(Q; \mathbb{R}) $, $ c_{ij}(t,x) \in W^{2,\infty}(Q; \mathbb{R}^n) $, and $ q_{ij}(t,x) \in L^{\infty}(Q; \mathbb{R})$, $1 \leq i,j \leq 2 $. 
\begin{proposition}\label{General-Wellposedness-system}
    Let $T>0$ be given and $\vec{\Phi}\in \mathbf{H}^{1}(\Omega), \vec{\Psi}\in \mathbf{L}^{2}(\Omega)$ and $\vec{f}\in \mathbf{H}^{1}(\Sigma)$ be such that the compatibility condition  $\vec{f}|_{t=0}=\vec{\Phi}|_{\partial\Omega}$ is valid. 
    Then there exists a unique solution $\vec{u}(t)$ to \eqref{IBVP for first order vector} such that
    \[
    \vec{u}(t)\in \mathbf{C}^{1}\lb [0,T];\mathbf{L}^{2}(\Omega)\rb\cap \mathbf{C}\lb [0,T];\mathbf{H}^{1}(\Omega)\rb, \quad \partial_{\nu}\vec{u}\in \mathbf{L}^{2}{(\Sigma)},
    \]
    and there exists a constant $C>0$ such that
\begin{equation}\label{Estimate for solution}
\begin{aligned}
\max_{0\leq t\leq T}\left(\lVert\vec{u}(t)\rVert_{\mathbf{H}^{1}(\Omega)}+\|\partial_{t}\vec{u}(t)\|_{\mathbf{L}^{2}(\Omega)}\right)+\lVert\partial_{\nu}\vec{u}\rVert_{\mathbf{L}^2(\Sigma)} \leq C \lb\lVert\vec{\Phi}\rVert_{\mathbf{H}^{1}(\Omega)}+ \lVert\vec{\Psi}\rVert_{\mathbf{L}^{2}(\Omega)}+\lVert\vec{f}\rVert_{\mathbf{H}^{1}(\Sigma)}\rb.
 \end{aligned}
	\end{equation}
\end{proposition}
\begin{proof}
We express the solution $\vec{u}$ as a sum of a homogeneous and a particular solution as $\vec{u}(t,x)=\vec{v}(t,x)+\vec{w}(t,x)$, where $\vec{v}$ and $\vec{w}$ solves
\begin{equation}\label{Equation for v}
	\begin{aligned}
		\begin{cases}
		\Box\vec{v}(t,x)=\vec{0}  \quad &\text { in } Q,\\
		\vec{v}(0,x)=\vec{\Phi}(x),\quad \PD_{t}\vec{v}(0,x)=\vec{\Psi}(x)  &\text { in }\Omega,\\
		\vec{v}(t,x)=\vec{f}(t,x)  \quad &\text { on } \Sigma,
		\end{cases}
		\end{aligned}
		\end{equation}
and
\begin{equation}\label{Equation for w}
\begin{aligned}
		\begin{cases}
\left(\Box + \mathbf{P}_{\mathbf{B},\mathbf{C},\mathbf{q}}\right)\vec{w} = - \mathbf{P}_{\mathbf{B},\mathbf{C},\mathbf{q}}\vec{v}   \quad  \text{ in } Q,\\
		\vec{w}(0,x)=\vec{0}, \quad \partial_t\vec{w}(0,x)=\vec{0} \quad \text { in }\Omega,\\
		\vec{w}(t,x)=\vec{0}  \quad\text { on } \Sigma.
		\end{cases}
		\end{aligned}
\end{equation}
Using the fact that $\vec{\Phi}\in \mathbf{H}^{1}(\Omega)$,  $\vec{\Psi}\in \mathbf{L}^{2}(\Omega)$, and $\vec{f}\in \mathbf{H}^{1}(\Sigma)$ from \cite[Theorem 2.30]{Lassas_Book} there exists a unique solution $\vec{v}(t,x)$ to \eqref{Equation for v} such that 
\begin{equation}\label{eq-1}
\begin{gathered}
\vec{v}(t)\in \mathbf{C}^{1}\lb [0, T];\mathbf{L}^{2}(\Omega)\rb\cap \mathbf{C}\lb [0,T];\mathbf{H}^{1}(\Omega)\rb, \quad \mbox{with}\quad \partial_{\nu}\vec{v}\in \mathbf{L}^{2}(\Sigma)\\
\mbox{and} \quad
\max_{0\leq t\leq T}\left(\sum_{k=0,1} \lVert\partial_t^k \vec{v}(t)\rVert_{\mathbf{H}^{1-k}(\Omega)}\right)+\lVert\partial_{\nu}\vec{v}\rVert_{\mathbf{L}^2(\Sigma)}\leq C \lb\lVert\vec{\Phi}\rVert_{\mathbf{H}^{1}(\Omega)}+ \lVert\vec{\Psi}\rVert_{\mathbf{L}^{2}(\Omega)}+\lVert\vec{f}\rVert_{\mathbf{H}^{1}(\Sigma)}\rb,
\end{gathered}
\end{equation}
where constant $C>0$ is independent of $\vec{v}$.
Furthermore, since $\mathbf{P}_{\mathbf{B},\mathbf{C},\mathbf{q}} \vec{v}\in \mathbf{L}^2(Q)$, then from \cite[Theorem 8.1, Chapter 3]{Lions_Magenese_Book} \eqref{Equation for w} admits a unique solution
\begin{equation}\label{eq-2}
\begin{gathered}
\vec{w}(t)\in \mathbf{C}^{1}\lb [0,T];\mathbf{L}^{2}(\Omega)\rb \cap \mathbf{C}\lb [0,T];\mathbf{H}^{1}(\Omega)\rb\\
\mbox{such that} \qquad 
\max_{0\leq t\leq T}\left(\lVert\vec{w}(t)\rVert_{\mathbf{H}^{1}(\Omega)}+\|\partial_{t}\vec{w}(t)\|_{\mathbf{L}^{2}(\Omega)}\right)\leq C\lb\lVert\vec{\Phi}\rVert_{\mathbf{H}^{1}(\Omega)}+ \lVert\vec{\Psi}\rVert_{\mathbf{L}^{2}(\Omega)}+\lVert\vec{f}\rVert_{\mathbf{L}^{2}(\Sigma)}\rb,
\end{gathered}
\end{equation}
where $C>0$ is independent of $\vec{w}$. Then, we apply arguments similar to \cite[Theorem 2.1]{RM-Matrix-potential} to obtain $\partial_{\nu}\vec{w}\in\mathbf{L}^2(\Sigma)$ with 
\begin{equation*}
\begin{aligned}
\lVert \partial_{\nu}\vec{w}\rVert_{\mathbf{L}^2(\Sigma)} \leq C\lb\lVert\vec{\Phi}\rVert_{\mathbf{H}^{1}(\Omega)}+ \lVert\vec{\Psi}\rVert_{\mathbf{L}^{2}(\Omega)}+\lVert\vec{f}\rVert_{\mathbf{H}^{1}(\Sigma)}\rb.
 \end{aligned}
\end{equation*}
We complete the proof by combining this with \eqref{eq-1} and \eqref{eq-2}.
\end{proof}

 \begin{corollary}\label{Cor_regularity}
Furthermore, to obtain the higher-order regularity for $\vec{u}$, we use the argument from \cite[Theorem 2.45]{Lassas_Book}. If $\vec{\Phi}\in \mathbf{H}^2(\O)$, $\vec{\Psi}\in \mathbf{H}^{1}(\O)$ and $\vec{f}\in \mathbf{H}^2(\Sigma)$ satisfies the compatibility conditions $\vec{f}|_{t=0}=\vec{\Phi}|_{\partial\O}$ and $\partial_{t}\vec{f}|_{t=0}=\partial_{t}\vec{u}|_{(\partial\O \times \{0 \})}=\vec{\Psi}|_{\partial\O}$. Then there exists a unique solution $\vec{u}(t)$ to \eqref{IBVP for first order vector} such that
 \[
 \vec{u}(t)\in \mathbf{C}\lb [0,T];\mathbf{H}^{2}(\Omega)\rb\cap \mathbf{C}^1\lb [0,T];\mathbf{H}^{1}(\Omega)\rb\cap \mathbf{C}^2\lb [0,T];\mathbf{L}^{2}(\Omega)\rb \quad \mbox { with } \quad \partial_{\nu}\vec{u}\in \bm{H}^{1}{(\Sigma)},
 \]
 and there exists a constant $C>0$ such that
\begin{equation}\label{Improved_Est_soln}
\begin{aligned}
\max_{0\leq t\leq T}\left(\sum_{k=0,1,2} \lVert\partial_t^k \vec{u}(t)\rVert_{\mathbf{H}^{2-k}(\Omega)}\right)
+\lVert\partial_{\nu}\vec{u}\rVert_{\mathbf{H}^1(\Sigma)} \leq C \lb\lVert\vec{\Phi}\rVert_{\mathbf{H}^{2}(\Omega)}+ \lVert\vec{\Psi}\rVert_{\mathbf{H}^{1}(\Omega)}+\lVert\vec{f}\rVert_{\mathbf{H}^{2}(\Sigma)}\rb.
 \end{aligned}
	\end{equation}
 \end{corollary}
\begin{proof}[Proof of the theorem  ~\ref{IBVP-Biwave-wellposedness}]
We denote $\vec{U}(t,x):=(u(t,x),\Box u(t,x))^{T}$. We can rewrite the first line of \eqref{IBVP} as the system of equations
\[
\Bigg \{\Box 
+\begin{bmatrix}
    0  & 0\\
    B(t,x)\partial_{t} & 0
\end{bmatrix}+
\begin{bmatrix}  
0  & 0\\   C(t,x)\cdot \nabla_{x} & 0
\end{bmatrix}
+
\begin{bmatrix}
0 & -1 \\
q(t,x) & A(t,x)
\end{bmatrix}\Bigg \}
\vec{U}=\vec{0}.
\]
Thus, solving \eqref{IBVP} is equivalent to solve
\begin{equation}
\begin{aligned}\label{Equivalent IBVP in terms of system}
\begin{cases}
\left(\Box+ \mathbf{B}(t,x,\partial_{t})+\mathbf{C}(t,x,\nabla_{x})+\mathbf{q}(t,x)\right)\vec{U}=\vec{0}&\text { in } Q,\\
(\vec{U}|_{t=0}, \partial_{t}\vec{U}|_{t=0})= \left(\psi_{0}, \psi_{2}-\Delta\psi_{0}, \psi_{1}, \psi_{3}-\Delta\psi_{1}\right)& \text { in } \Omega,\\
\vec{U}=\left(f,g\right)& \text { on } \Sigma,
	 \end{cases}
	 \end{aligned} 
  \end{equation}
where 
 \[
 \mathbf{B}(t,x,\partial_{t})=\begin{bmatrix}
     0  & 0\\
    B(t,x)\partial_{t}  & 0
 \end{bmatrix}, \quad 
 \mathbf{C}(t,x,\nabla_{x})=\begin{bmatrix}
     0  & 0\\
    C(t,x)\cdot \nabla_{x}  & 0
 \end{bmatrix}, 
 \]
 and
 \[
 \mathbf{q}(t,x)
 =\begin{bmatrix}
 0 & -1 \\
 q(t,x) & A(t,x)
 \end{bmatrix}.
 \]
Let $\psi_{j} \in H^{4-j}(\Omega)$ for $j=0,1,2,3$, $f\in H^{4}(\Sigma)$ and $g\in H^{2}(\Sigma)$ satisfy the compatibility condition
\[
(\psi_{0},\psi_{2}-\Delta_{x}\psi_{0})|_{\partial\Omega}= (f,g)|_{(\{t=0\}\times \partial\Omega)}
\mbox{ and } (\psi_{1},\psi_{3}-\Delta_{x}\psi_{1})|_{\partial\Omega}= \partial_{t}(f,g)|_{(\{t=0\}\times \partial\Omega)}.    
\]
Using Corollary~\ref{Cor_regularity} applied to \eqref{Equivalent IBVP in terms of system}, we obtain the existence of a unique solution 
\[ \vec{U}=
\begin{pmatrix} u \\ \Box u \end{pmatrix}\in \mathbf{C}\lb [0,T];\mathbf{H}^{2}(\Omega)\rb\cap \mathbf{C}^1\lb [0,T];\mathbf{H}^{1}(\Omega)\rb\cap \mathbf{C}^2\lb [0,T];\mathbf{L}^{2}(\Omega)\rb,
\]
satisfying the estimate \eqref{Energy_Est-1}.
\end{proof}

\section{Construction of Geometric optics solutions}\label{GO-Soln}
In this section, we give an explicit construction of a class of exponentially growing and a class of exponentially decaying solutions of the perturbed bi-wave operator $\mathscr{L}_{A, B, C, q}$ and its $L^2$-adjoint $\mathscr{L}_{A, B, C, q}^{\ast}$, respectively. We begin with an interior Carleman estimate for the operator $\mathscr{L}_{A,B,C,q}$, which acts as an important ingredient in the construction of solutions. To this end, we use notations from semiclassical Sobolev spaces. For any $h > 0$, a small parameter, and for any non-negative integer $m$, we define the semiclassical Sobolev space of order $m \in \mathbb{N}$ by $H^m_{\mathrm\mathrm{scl}}(Q)$, which is the space $H^m(Q)$ endowed with the semiclassical norm:
\[
 \lVert u \rVert^2_{H^m_{\mathrm{scl}}(Q)} = \sum_{|\alpha| \le m} \lVert ( h D)^{\alpha} u \rVert^2_{L^2(Q)}, 
 \quad \mbox{where } D := \nabla_{t,x}.
\]
Similar to the Sobolev spaces, for $m\geq 0$ we define $H^{-m}_{\mathrm{scl}}(\R^n)$ as the dual of the Sobolev space $H^{m}_{\mathrm{scl}}(\R^n)$.

\begin{proposition}[Interior Carleman estimate] \label{Int_Carleman_est}
Let $\vp(t,x):=t+x\cdot \omega$, where $\omega\in \mathbb{S}^{n-1}$ is fixed and $\mathscr{L}_{\vp}:= h^{4}e^{-\frac{\vp}{h}}\mathscr{L}_{A,B,C,q}e^{\frac{\vp}{h}}$. Then there exists $h_{0}>0$ such that 
    \begin{align}\label{Interior Carleman estimate}
        \lVert u\rVert_{L^{2}(\R^{1+n})} \leq \frac{C}{h^{2}}\lVert \mathscr{L}_{\vp} u\rVert_{H_\mathrm{scl}^{-2}(\R^{1+n})},
    \end{align}
   for all $u\in C_{c}^{\infty}(Q)$, $0<h\leq h_{0}$.
\end{proposition}
\begin{proof}
We first prove the estimate for the principal part of $\mathscr{L}$, i.e., $\Box^2$, and then we add the lower order terms into it.
Let $0<\ve<1$ be any real number. Consider the modified Carleman weight (see \cite{ferreira2007determining,Kian_Damping_2016,K.Vashisth})
\[
\vp_{\ve}(t,x)= \vp(t,x)-\frac{ht^{2}}{2\ve},\quad \text { in } Q.
\]
We claim for $0<h\ll\ve\ll 1$ and $s\in\R$, we have
\begin{equation}\label{Carleman_est_general_s}
\frac{h}{\sqrt{\ve}}\lVert u \rVert_{H_\mathrm{scl}^{s+1}(\R^{1+n})}\leq C \lVert e^{-\frac{\vp_{\ve}}{h}}(h^{2}\Box) e^{\frac{\vp_{\ve}}{h}}u\rVert_{H_\mathrm{scl}^{s}(\R^{1+n})}, \quad C>0,
 \end{equation}
for all $u\in C_{c}^{\infty}(Q)$.
For $s=0$, the estimate  \eqref{Carleman_est_general_s} has been proved in \cite{K.Vashisth}. We prove \eqref{Carleman_est_general_s} for any $s\in \R$ by shifting the Sobolev index, which is described as follows. Let $\widetilde{Q}$ be a bounded open set containing $Q$ in its interior, and let $\chi\in C_{c}^{\infty}(\widetilde{Q})$ with $\chi =1$ on $Q$. For any $u\in C_{c}^{\infty}(Q)$, let $v=\chi\langle hD\rangle^{s}u$, then using the estimate in \cite{K.Vashisth} for $v \in C_{c}^{\infty}(\wt{Q})$ we obtain
     \begin{equation}\label{E-4}
         \frac{h}{\sqrt{\varepsilon}}\|v\|_{H_\mathrm{scl}^{1}(\R^{1+n})}\leq C \|P_{\varphi_{\varepsilon}}v\|_{L^{2}(\R^{1+n})}, \quad \text{ where }\quad  P_{\vp_{\ve}}:= e^{-\frac{\vp_{\ve}}{h}}(h^{2}\Box) e^{\frac{\vp_{\ve}}{h}}.
    \end{equation}
Using continuity properties of pseudo-differential operators on semiclassical Sobolev spaces (see \cite{Martinez}), we obtain
\begin{equation}\label{Eq_5}
   \begin{aligned}
       h\|u\|_{H_{scl}^{s+1}(\R^{1+n})}&=h\| \langle hD\rangle^{s} \chi u\|_{H_{scl}^{1}(\R^{1+n})}\\
     &\leq h\|\chi \langle hD\rangle^{s}u\|_{H_{scl}^{1}(\R^{1+n})}+h\|\commutator{\chi} {\langle hD\rangle^{s}}u\|_{H_{scl}^{1}(\R^{1+n})}\\
     & \leq \Oc({\sqrt{\ve}})\|P_{\vp_{\ve}}\chi \langle hD\rangle^{s}u\|_{L^{2}(\R^{1+n})}+Ch^{2}\|E_{s-1}u\|_{H_{scl}^{1}(\R^{1+n})}\\
     &\leq \Oc(\sqrt{\ve})\left( \| \langle hD\rangle^{s}(P_{\vp_{\ve}}\chi) u\|_{L^{2}(\R^{1+n})}+ \| \commutator{P_{\vp_{\ve}}\chi}{\langle hD\rangle^{s}}u\|_{L^{2}(\R^{1+n})}\right)+ Ch^{2}\|u\|_{H_{scl}^{s}(\R^{1+n})}\\
     &\leq \Oc(\sqrt{\ve})\|P_{\vp_{\ve}}u\|_{H_{scl}^{s}(\R^{1+n})}+\Oc(\sqrt{\ve})h\|E_{s+1}u\|_{L^{2}(\R^{1+n})}+Ch^{2}\|u\|_{H_{scl}^{s+1}(\R^{1+n})}\\
&\leq\Oc(\sqrt{\ve})\|P_{\vp_{\ve}}u\|_{H_{scl}^{s}(\R^{1+n})}+\Oc(\sqrt{\ve})h\|u\|_{H_{scl}^{s+1}(\R^{1+n})}+Ch^{2}\|u\|_{H_{scl}^{s+1}(\R^{1+n})},
    \end{aligned}
\end{equation}
where $E_{s-1}$ and $E_{s+1}$ are pseudo-differential operators of order $s-1$ and $s+1$ respectively.
By absorbing the extra terms $\|u\|_{H_\mathrm{scl}^{s+1}}$ in the left-hand side of \eqref{Eq_5}, for all $0<h\ll \ve\ll 1$, we get
    \[\frac{h}{\sqrt{\ve}}\lVert u \rVert_{H_\mathrm{scl}^{s+1}(\R^{1+n})}
    \leq C \lVert P_{\vp_{\ve}} u \rVert_{H_\mathrm{scl}^{s}(\R^{1+n})}, \quad \text { for all $u\in C_{c}^{\infty}(Q)$.}
    \]
Using the estimate \eqref{Carleman_est_general_s} for 
$s=-2$ and $s=-1$ we get 
\begin{equation}\label{Car_Iteration}
\lVert \left(P_{\vp_{\ve}}\right)^2 u \rVert_{H_\mathrm{scl}^{-2}(\R^{1+n})}
\geq C\frac{h}{\sqrt{\ve}}\lVert P_{\vp_{\ve}} u \rVert_{H^{-1}(\R^{1+n})}
\geq C\frac{h^{2}}{\ve}\lVert u\rVert_{L^{2}(\R^{1+n})},
\end{equation}
for all $ u\in C_{c}^{\infty}(Q)$. For the perturbed bi-wave operator $\mathscr{L}_{A,B,C,q}$ defined in \eqref{operator}, we introduce the conjugated operator
\[
\mathscr{L}_{\varphi_\ve}:=h^4 e^{-\frac{\varphi_\ve}{h}} \mathscr{L}_{A,B,C,q} e^{\frac{\varphi_\ve}{h}},
\]
where $\varphi_\ve(t,x) = t+x\cdot\omega-\frac{h t^2}{2\ve}$.
Now considering the conjugated operator $\mathscr{L}_{\vp_{\ve}}$, we have
\begin{equation}\label{Car_est_1}
\begin{aligned}
     \lVert \mathscr{L}_{\vp_{\ve}}u\rVert_{H_\mathrm{scl}^{-2}} 
    \geq &\lVert 
\left(P_{\vp_{\ve}}\right)^2u\rVert_{H_\mathrm{scl}^{-2}}-\lVert e^{-\frac{\vp_{\ve}}{h}}h^{4}A\Box \left(e^{\frac{\vp_{\ve}}{h}}u\right)\rVert_{H_\mathrm{scl}^{-2}}
-\lVert e^{-\frac{\vp_{\ve}}{h}}h^{4}B\partial_{t}\left(e^{\frac{\vp_{\ve}}{h}}u\right)\rVert_{H_\mathrm{scl}^{-2}}\\
&-\lVert e^{-\frac{\vp_{\ve}}{h}}h^{4}C\cdot \nabla_{x}\left(e^{\frac{\vp_{\ve}}{h}}u\right)\rVert_{H_\mathrm{scl}^{-2}}
    -\lVert e^{-\frac{\vp_{\ve}}{h}}h^{4}qe^{\frac{\vp_{\ve}}{h}}u\rVert_{H_\mathrm{scl}^{-2}}.
    \end{aligned}
\end{equation}
Note that, 
\[
\lVert e^{-\frac{\vp_{\ve}}{h}}h^{4}qe^{\frac{\vp_{\ve}}{h}}u\rVert_{H_\mathrm{scl}^{-2}(\R^{n+1})} 
\leq h^4 \| qu\|_{L^2(Q)}
\leq h^4 \|q\|_{L^{\infty}(Q)}\|u\|_{L^2(Q)}
\leq ch^4\|u\|_{L^2(Q)},
\]
since $q \in L^{\infty}(Q)$ and $u \in C^{\infty}_c(Q)$.
Moreover, $C \in W^{1,\infty}(Q)$ ensures
\begin{equation*}
\begin{aligned}
\lVert e^{-\frac{\vp_{\ve}}{h}}h^{4}C\cdot \nabla_{x}e^{\frac{\vp_{\ve}}{h}}u\rVert_{H_\mathrm{scl}^{-2}(\R^{1+n})} 
&\leq h^3 \| h\nabla_x \cdot(C u) \|_{H_\mathrm{scl}^{-2}(\R^{1+n})}
+ h^3 \| (h\nabla_x \cdot (e^{-\frac{\vp_{\ve}}{h}}C))  u \|_{H_\mathrm{scl}^{-2}(\R^{1+n})}\\
&\leq h^3\left( \| Cu\|_{L^2(Q)} + \| (h\nabla_x \cdot (e^{-\frac{\vp_{\ve}}{h}}C))\|_{L^{\infty}(Q)} \|u\|_{L^2(Q)}\right)\\
&\leq ch^{3}\|u\|_{L^{2}(\R^{1+n})}.
\end{aligned}
\end{equation*}
Likewise, since $A \in W^{2,\infty}(\R^{n+1})$ and $B \in W^{1,\infty}(Q)$, we calculate
\[\begin{aligned}
\lVert e^{-\frac{\vp_{\ve}}{h}}h^{4}A\Box e^{\frac{\vp_{\ve}}{h}}u\rVert_{H_\mathrm{scl}^{-2}(\R^{1+n})}
+ \lVert e^{-\frac{\vp_{\ve}}{h}}h^{4}B\partial_{t}e^{\frac{\vp_{\ve}}{h}}u\rVert_{H_\mathrm{scl}^{-2}(\R^{1+n})} 
\leq Ch^{2}\|u\|_{L^{2}(\R^{1+n})}.
\end{aligned}\]
Now, combining the above estimates with \eqref{Car_Iteration} and \eqref{Car_est_1}, for $h\ll \ve\ll 1$ small enough, we obtain
\begin{equation*}
    \lVert \mathscr{L}_{\vp_{\ve}}u\rVert_{H_\mathrm{scl}^{-2}(\R^{1+n})}
    \geq c\frac{h^{2}}{\ve}\|u\|_{L^{2}(Q)}-ch^{2}\|u\|_{L^{2}(Q)}
    \geq c\frac{h^{2}}{\ve}\|u\|_{L^{2}(Q)}. 
\end{equation*}
For a fixed $\ve>0$ there exists $M>0$ such that $1\leq e^{\frac{t^{2}}{2\ve}}\leq M$. This proves the Carleman estimate \eqref{Interior Carleman estimate}.
\end{proof}
We denote the formal $L^2$ adjoint of $\mathscr{L}_{\varphi}$ as $\mathscr{L}_{\varphi}^{\ast}$ satisfying $\langle \mathscr{L}_{\varphi}^{\ast}u,v\rangle_{L^2(Q)} = \langle u,\mathscr{L}_{\varphi}v\rangle_{L^2(Q)}$,
\[
\mathscr{L}_{\varphi}^{\ast} = h^4 e^{\frac{\varphi}{h}} \mathscr{L}_{A, B, C, q}^{\ast} e^{\frac{-\varphi}{h}}.
\]
Now we compute the $L^2$ adjoint of the operator $\mathscr{L}_{A,B,C,q}$. For smooth $u,v$ compactly supported in $Q$, integration by parts yields
\[
\int_Q (\mathscr{L}_{A,B,C,q} u)v = \int_Q u\bigl(\Box^2 v + \widetilde{A}\Box v + \widetilde{B}\partial_t v + \widetilde{C}\cdot\nabla_x v + \widetilde{q} v\bigr),
\]
where
\[
\widetilde{A}=A,\quad \widetilde{B}=B+2\partial_t A,\quad \widetilde{C}=C-2\nabla_x A,\quad \widetilde{q}=\Box A+\partial_t B + \nabla_x\cdot C + q.
\]
Hence the formal $L^2$ adjoint is
\[
\mathcal{L}_{A,B,C,q}^* = \square^2 + \widetilde{A}\square + \widetilde{B}\partial_t + \widetilde{C}\cdot\nabla_x + \widetilde{q},
\]
with 
\begin{equation}\label{Adjoint_coeffs}
\begin{aligned}
\widetilde{A}(t,x) &:= A(t,x), \\
\widetilde{B}(t,x) &:= B(t,x) + 2\partial_t A(t,x), \\
\widetilde{C}(t,x) &:= C(t,x) - 2\nabla_x A(t,x), \\
\widetilde{q}(t,x) &:= \Box A(t,x) + \partial_t B(t,x) + \nabla_x \cdot C(t,x) + q(t,x).
\end{aligned}\end{equation}
Note that $\mathscr{L}_{A, B, C, q}^{\ast}$ has a similar form to $\mathscr{L}_{A, B, C, q}$ with $\wt{A}\in W^{3,\infty}(Q)$, $\wt{B},\wt{C}\in W^{2,\infty}(Q)$ and $\wt{q}\in L^{\infty}(Q)$. Therefore, the Carleman estimate derived in Proposition~\ref{Int_Carleman_est} also holds for $\mathscr{L}_{\varphi}^{\ast}$, since $-\varphi$ is also a Carleman weight. To construct geometric optics solutions, we use the following solvability result. The proof is based on standard functional analysis arguments analogous to those used in \cite{Kian_Damping_2016,K.Vashisth}.

\begin{proposition}\label{Existence of solution for conjugated operator}
Let $\vp$ be as defined in Proposition \ref{Interior Carleman estimate} and $\mathscr{L}_{A,B,C,q}$ as \eqref{operator}. For $h>0$ small enough and $v\in L^{2}(Q),$ there exists a solution $u\in H^{2}(Q)$ of 
 \[
 e^{-\frac{\vp}{h}}h^{4}\mathscr{L}_{A,B,C,q}e^{\frac{\vp}{h}}u=v, \text { in } Q,
 \]
 and satisfying
\[
\|u\|_{H_\mathrm{scl}^{2}(Q)} \leq \frac{C}{h^{2}}\|v\|_{L^{2}(Q)},
 \]
where $C>0$ is a constant independent of $h$.
\end{proposition}

Next, using Proposition \ref{Existence of solution for conjugated operator}, we construct geometric optics solutions to the equation 
\begin{equation}\label{biwave-eqn}
\mathscr{L}_{A,B,C,q} u_g=0 \quad \text { in } Q.
\end{equation}

\begin{proposition}\label{Existence of growing solution}  
Let $\vp$ be as defined in Proposition \ref{Interior Carleman estimate} and $\mathscr{L}_{A, B, C,q}$ as \eqref{operator}. There exists $h_{0}>0$ and $r_{g} \in H^2(Q)$ such that \begin{equation}\label{decay_est}
\|r_{g}\|_{H_\mathrm{scl}^{2}(Q)}=\mathcal{O}(h^2),
\end{equation}
and
\begin{equation}\label{Exponential growing solution form}
u_{g}(t,x)= e^{\frac{\vp(t,x)}{h}}\left( a_{g_{0}}(t,x)+ha_{g_{1}}(t,x)+r_{g}(t,x;h)\right), 
\end{equation}
is a solution of the equation $\mathscr{L}_{A,B,C,q}u=0$, when $0<h\leq h_{0}$, and $a_{g_{0}}\in C^{\infty}(\Bar{Q})$, $a_{g_{1}} \in W^{2,\infty}(Q)$ solves the transport equations \eqref{First-Transport-eqn} and \eqref{Second-Transport-eqn} , respectively.
\end{proposition}
\begin{proof}
We re-write the conjugated operator $e^{\frac{-\vp}{h}}h^{4}\mathscr{L}_{A,B,C,q}e^{\frac{\vp}{h}}$ as
\begin{equation*}
\begin{aligned}
e^{\frac{-\vp}{h}}h^{4}\mathscr{L}_{A,B,C,q}e^{\frac{\vp}{h}}
& =\left( h^{2}\Box + 2hT\right)^{2}+h^{2}A\left( h^{2}\Box + 2hT\right)+ h^{4}B\partial_{t}+h^{4}C\cdot \nabla \\
&\quad \quad \quad +h^{3}(B+C\cdot \omega)+h^{4}q,
\end{aligned}
\end{equation*}
where $T=(\PD_{t}-\omega\cdot\nabla_{x})$. Now substituting \eqref{Exponential growing solution form} into $h^4\mathscr{L}_{A,B,C,q}u_{g}(t,x)=0$ and by equating the terms of different powers of $h$ to zero, in $Q$ we get the following equations:
\begin{align}
    T^{2}a_{g_{0}}(t,x)=&0,  \label{First-Transport-eqn}\\
    T^{2}a_{g_{1}}(t,x)=& \lb-\frac{1}{2}(\Box\circ T+ T\circ\Box)-\frac{1}{4}B-\frac{1}{4}C\cdot \omega-\frac{1}{2}A T\rb a_{g_{0}}, \label{Second-Transport-eqn}\\
    e^{\frac{-\vp}{h}}h^{4}\mathscr{L}_{A,B,C,q}(e^{\frac{\vp}{h}}r_{g})
   =&-e^{-\frac{\vp}{h}}h^{4}\mathscr{L}_{A,B,C,q}\left(e^{\frac{\vp}{h}}(a_{g_{0}}+ha_{g_{1}})\right). \label{eqn-for-remainder}
\end{align}
To solve \eqref{First-Transport-eqn} and \eqref{Second-Transport-eqn}, we first find $v\in W^{2,\infty}(Q)$ which satisfies 
\begin{equation}\label{Transport_sys1}
    Tv=g \quad \text { in } Q ,\quad \text { for given } g\in W^{2,\infty}(Q);
\end{equation}
and then solve 
\begin{equation}\label{Transport_sys2}
    Ta_{g_{1}}=v \quad \text { in } Q.
\end{equation}
Following \cite{Evans-PDE}, given $g\in W^{2,\infty}(Q)$, we can solve \eqref{Transport_sys1} as
\[
v(t,x)=\int_{0}^{t}g(s,x+\omega(t-s))\D s + f(x+t\omega), \quad \text { for any } f\in W^{2,\infty}(\R^n).
\]
Similarly, the equation \eqref{Transport_sys2} can be solved; hence \eqref{Second-Transport-eqn} is solvable.
Having chosen $a_{g_{0}}\in C^{\infty}(\bar{Q})$ and $a_{g_{1}}\in W^{2,\infty}(Q)$, from \eqref{eqn-for-remainder} we get
\[
\|e^{\frac{-\vp}{h}}h^{4}\mathscr{L}_{A,B,C,q}(e^{\frac{\vp}{h}}r_{g})\|_{L^2(Q)}=\mathcal{O}(h^4).
\]
Now using Proposition \ref{Existence of solution for conjugated operator}, for $h>0$ small enough, there exists solution $r_{g}\in H^{2}(Q)$ of \eqref{eqn-for-remainder} with $\|r_{g}\|_{H_\mathrm{scl}^{2}(Q)}=\mathcal{O}(h^2)$.
\end{proof}
Similarly, we can construct geometric optic solutions for the formal $L^2$-adjoint operator $\mathscr{L}_{A, B, C,q}^{\ast}$ in $Q$.

\begin{corollary}\label{Decaying_soln}
Let $\mathscr{L}_{A,B,C,q}^{\ast}$ be as defined above and $\vp(t,x)=t+x\cdot \omega$, where $\omega\in\Sb^{n-1}$ is fixed. There exists $h_{0}>0$ and $r_{d}$ such that $\|r_{d}\|_{H_\mathrm{scl}^{2}(Q)}=\mathcal{O}(h^2)$ and
\begin{equation}\label{Exponential decaying solution form}
v_{d}(t,x)= e^{-\frac{\vp(t,x)}{h}}\left( b_{d_{0}}(t,x)+hb_{d_{1}}(t,x)+r_{d}(t,x;h)\right), 
 \end{equation}
is a solution of the equation $\mathscr{L}_{A,B,C,q}^{\ast}v=0$, when $0<h\leq h_{0}$, and $b_{d_{0}}$, $b_{d_{1}}$ satisfies the following transport equations:
\begin{equation}\label{Transport for adjoint eqns}
\begin{gathered}
T^2 b_{d_0}(t,x) = 0 \quad \text{in } Q, \quad \text{where} \quad T := (\partial_t - \omega \cdot \nabla),\\
T^2 b_{d_1}(t,x) = \left( -\frac{1}{2} (\Box \circ T + T \circ \Box) - \frac{1}{4} \wt{B} - \frac{1}{4} \wt{C} \cdot \omega - \frac{1}{2} \wt{A} T \right) b_{d_0}(t,x) \quad \text{in}\quad Q.
\end{gathered}
\end{equation}
\end{corollary}
\begin{remark}\label{Rem_regularity}
It is clear from the above that $b_{d_1} \in H^2(Q)$ provided $\widetilde{A},\widetilde{B},\widetilde{C} \in W^{2,\infty}(Q)$. On the other hand, from \eqref{Adjoint_coeffs} we need
$A \in W^{3,\infty}(Q)$ and $B,C \in W^{2,\infty}(Q)$ to achieve the above regularities for $\widetilde{A},\widetilde{B},\widetilde{C}$. Since we work with $v_d$ in $H^2(Q)$, this justifies our regularity assumptions for the coefficients in Theorem~\ref{Main-Result}.
\end{remark}

\section{Proof of main Theorem}\label{Proof_main_thm}
In this section, we prove the unique recovery of the coefficients, that is, Theorem \ref{Main-Result}. We use the special solutions constructed in the previous section to obtain Light Ray Transforms for the coefficients. We prove a boundary Carleman estimate, which helps us to estimate the boundary terms appear due to partial measurements. In the final step, we use injectivity results of LRT to complete the proof.
\subsection{Boundary Carleman estimate}
We prove a boundary Carleman estimate for the operator $\mathscr{L}_{A,B,C,q}$.
\begin{proposition}
Let $\vp(t,x)=t+x\cdot \omega$ where $\omega\in \mathbb{S}^{n-1}$ is fixed. Let $u\in C^{4}(\overline{Q})$ such that 
\begin{equation}\label{Intial_boundary_0}
    \partial^{l}_{t}u|_{t=0}=0, \quad \text{for } l=0,1,2,3 \quad \text{and}\quad  u|_{\Sigma}=(\Box u)|_{\Sigma}=0.
\end{equation}
If $A$, $B$, $C$ and $q$ are as in Theorem \ref{Main-Result} and $\Sigma_{\pm,\omega}$ be as defined in \eqref{lateral_bddy_part}, then, there exists $h_0>0$ such that for $0<h<h_0$ we have
\begin{equation}\label{Boundary-Carleman-estimate}
\begin{aligned}
 &\lVert e^{-\vp/h}h^{4}\mathscr{L}_{A,B,C,q} u\rVert ^{2}_{L^{2}(Q)}
 +\sum_{k,l=0}^{1}h^{4+l}\lVert e^{-\frac{\vp(T,\cdot)}{h}}(\nabla_{x})^{l}\Box^k u(T,\cdot)\rVert^{2}_{L^{2}(\Omega)}\\
 &
 +\sum_{l=0}^{1}h^{5-2l} \left(
\lVert e^{-\vp/h} (-\partial_{\nu}\vp)^{\frac{1}{2}} \partial_{\nu}(h^{2}\Box)^l u\rVert^{2}_{L^{2}(\Sigma_{-,\omega})}\right)\\
  &\geq \frac{1}{C}\biggl\{\sum_{l=0}^{1} h^{4-2l}\lVert e^{-\vp/h} (h^{2}\Box)^l u\rVert^{2}_{H_\mathrm{scl}^{1}(Q)}\\
&\quad \quad +\sum_{l=0}^{1} h^{5-2l} \left( 
\lVert e^{\frac{-\vp(T,\cdot)}{h}}\partial_{t}(h^{2}\Box)^l u(T,\cdot)\rVert^{2}_{L^{2}(\Omega)}
+\lVert(\partial_{\nu}\vp)^{\frac{1}{2}} e^{-\vp/h}\partial_{\nu}(h^{2}\Box)^l u\rVert^{2}_{L^{2}(\Sigma_{+,\omega})} 
\right)\biggl\}.
   \end{aligned} 
   \end{equation}
   \end{proposition}

\begin{proof}
To prove the estimate \eqref{Boundary-Carleman-estimate}, we begin by recalling the boundary Carleman estimate for the conjugated semiclassical wave operator $(h^2 e^{-\vp/h}\Box e^{\vp/h})$ from \cite[Theorem 3.1]{K.Vashisth} as
\begin{equation}\label{Boundary-Carleman-Est} 
 \begin{aligned}
 &\lVert  e^{-\vp/h}h^{2}\Box u\rVert ^{2}_{L^{2}(Q)}+h^{2}\left\langle e^{-\frac{\vp(T,\cdot)}{h}}u(T,\cdot),e^{-\frac{\vp(T,\cdot)}{h}}u(T,\cdot)\right\rangle_{L^{2}(\Omega)}
 \\
& +h^{3}\left\langle e^{-\frac{\vp(T,\cdot)}{h}}\nabla_{x}u(T,\cdot),e^{-\frac{\vp(T,\cdot)}{h}}\nabla_{x}u(T,\cdot)\right\rangle_{L^{2}(\Omega)}+h^{3}\left\langle  e^{-\vp/h}\lb-\partial_{\nu}\vp\rb\partial_{\nu}u,e^{-\vp/h}\partial_{\nu}u\right\rangle_{L^{2}\lb\Sigma_{-,\omega}\rb}\\
&\geq \frac{1}{C}\biggl\{
  h^{2}\lVert e^{-\vp/h}u\rVert^{2}_{H_\mathrm{scl}^{1}(Q)}+ h^{3}\left\langle  e^{-\vp/h}(\partial_{\nu}\vp)\partial_{\nu}u,e^{-\vp/h}\partial_{\nu}u\right\rangle_{L^{2}(\Sigma_{+,\omega})}
  \\
 &\hspace{5cm} +h^{3}\left\langle e^{-\frac{\vp(T,\cdot)}{h}}\partial_{t}u(T,\cdot),e^{-\frac{\vp(T,\cdot)}{h}}\partial_{t}u(T,\cdot)\right\rangle_{L^{2}(\Omega)}
  \biggl\},
\end{aligned}
 \end{equation}
holds for all $u\in C^{2}(\overline{Q})$ with $u|_{\Sigma}=0,\ u|_{t=0}=\partial_{t}u|_{t=0}=0,$ and $h>0$ small enough. 
Replacing $u$ with $(h^{2}\Box u)$ for $u\in C^{4}(\overline{Q})$ satisfying \eqref{Intial_boundary_0} in \eqref{Boundary-Carleman-Est}, we obtain
\begin{equation}\label{bdy_car_est}
\begin{aligned}
 & \lVert  e^{-\vp/h}(h^{2}\Box)^{2} u\rVert ^{2}_{L^{2}(Q)} +\sum_{l=0}^{1}h^{2+l}\lVert e^{-\frac{\vp(T,\cdot)}{h}}(\nabla_{x})^{l}(h^{2}\Box) u(T,\cdot)\rVert^{2}_{L^{2}(\Omega)}\\
 &+\sum_{l=0}^{1}h^{4+l}\lVert e^{-\frac{\vp(T,\cdot)}{h}}(\nabla_{x})^{l}u(T,\cdot)\rVert^{2}_{L^{2}(\Omega)}+\sum_{l=0}^{1}h^{5-2l} \left(
\lVert e^{-\vp/h} (-\partial_{\nu}\vp)^{\frac{1}{2}} \partial_{\nu}(h^{2}\Box)^l u\rVert^{2}_{L^{2}(\Sigma_{-,\omega})}\right)\\
  &\geq \frac{1}{C} \frac{h^{2}}{2}\lVert e^{-\vp/h}(h^{2}\Box) u\rVert^{2}_{H_\mathrm{scl}^{1}(Q)} + 
  \frac{1}{2C^2}\biggl\{h^{4}\lVert e^{-\vp/h}u\rVert^{2}_{H_\mathrm{scl}^{1}(Q)}+
  +h^{3}\lVert(\partial_{\nu}\vp)^{\frac{1}{2}}e^{-\vp/h}\partial_{\nu}(h^{2}\Box)u\rVert^{2}_{L^{2}(\Sigma_{+,\omega})}\\
  &+h^{5}\lVert(\partial_{\nu}\vp)^{\frac{1}{2}}e^{-\vp/h}\partial_{\nu}u\rVert^{2}_{L^{2}(\Sigma_{+,\omega})}+h^{5}\lVert e^{\frac{-\vp(T,\cdot)}{h}}\partial_{t}u(T,\cdot)\rVert^{2}_{L^{2}(\Omega)}+h^{3}\lVert e^{-\frac{\vp(T,\cdot)}{h}}\partial_{t}(h^{2}\Box) u(T,\cdot)\rVert^{2}_{L^{2}(\Omega)}\biggl\},
  \end{aligned}
  \end{equation}
holds true for all $u\in C^{4}(\overline{Q})$ with \eqref{Intial_boundary_0}, $h>0$ small enough and the constant $C$ is independent of $u$ and $h$.
Now, we estimate the lower order perturbations as
\begin{equation*}
    \begin{aligned}
  &\lVert e^{\frac{-\vp}{h}}h^{2}A(h^{2}\Box) u\rVert_{L^{2}(Q)}\leq \mathcal{O}(h^{2})\lVert e^{\frac{-\vp}{h}}
  (h^{2}\Box)u\rVert_{H_\mathrm{scl}^{1}(Q)}, \quad \lVert e^{\frac{-\vp}{h}}h^{4}B\partial_{t}u\rVert_{L^{2}(Q)}\leq \mathcal{O}(h^{3})\lVert e^{\frac{-\vp}{h}}u\rVert_{H_\mathrm{scl}^{1}(Q)},\\
  &\lVert e^{\frac{-\vp}{h}}h^{4}C\cdot \nabla_{x}u\rVert_{L^{2}(Q)}
\leq \mathcal{O}(h^{3})\lVert e^{\frac{-\vp}{h}}u\rVert_{H_\mathrm{scl}^{1}(Q)}, \quad \lVert e^{\frac{-\vp}{h}}h^{4}q u\rVert_{L^{2}(Q)} \leq \mathcal{O}(h^{4})\lVert e^{\frac{-\vp}{h}}u\rVert_{H_\mathrm{scl}^{1}(Q)}.
\end{aligned}
\end{equation*}
Hence, by adding the lower order perturbations of in \eqref{bdy_car_est}, we obtain the required estimate \eqref{Boundary-Carleman-estimate}.
\end{proof}

\begin{remark} By a density argument, the Carleman estimate \eqref{Boundary-Carleman-estimate} can be extended to any function 
\begin{equation*}     
u \in \mathcal{V} := \Biggl\{u \in \bigcap_{k=0}^2 C^k\big([0,T]; H^{2-k}(\Omega)\big) \, :
\begin{array}{l}
\Box u \in \bigcap_{k=0}^2 C^k\big([0,T]; H^{2-k}(\Omega)\big),\\
\Box^2 u \in L^2(Q) \text{ and } u \text{ satisfies } \eqref{Intial_boundary_0}
\end{array}
\Biggr\}.     
\end{equation*}
If $f := \Box^2 u \in L^2(Q)$, then we can approximate it by a sequence $f_i \in C_c^\infty(Q)$ such that $f_i \to f$ in $L^2(Q)$ as $i \to \infty$. Let $u_i$ be the solution to $\Box^2 u_i = f_i$ satisfying the boundary and initial conditions  \eqref{Intial_boundary_0}. Then, by standard regularity theory, $u_i \in C^\infty(\overline{Q})$, and the Carleman estimate \eqref{Boundary-Carleman-estimate} holds for each $u_i$. Finally, using the energy estimate \eqref{Energy_Est-1}, we conclude that the Carleman estimate \eqref{Boundary-Carleman-estimate} also holds for $u\in \mathcal{V}$. 
\end{remark}

\subsection{Integral identity involving the coefficients}
In this subsection, we derive an integral identity that relates the difference of coefficients to the difference of the associated partial input-output operator. 
Let $u_{i}$ for $i=1,2$, be the solution to the following IBVPs
\begin{equation}
\begin{aligned}\label{Equation for ui}
\begin{cases}
\mathscr{L}_{A^{(i)},B^{(i)},C^{(i)}, q^{(i)}}u_{i} =0 & \text { in } Q,\\
\left (u_{i}|_{t=0}, \partial_{t}u_{i}|_{t=0}, \partial_{t}^{2}u_{i}|_{t=0}, \partial_{t}^{3}u_{i}|_{t=0}\right)=\left(\psi_{0},\psi_{1},\psi_{2},\psi_{3}\right) &\text { on }\Omega,\\
\left( u_{i}|_{\Sigma}, (\Box u_{i})|_{\Sigma}\right)= \left ( f, g\right) &\text { on } \Sigma.
\end{cases}
\end{aligned} 
\end{equation}
Let $u(t,x):=(u_{1}-u_{2})(t,x)$ then
\begin{equation}
\begin{aligned}\label{Equation for u}
\begin{cases}
\mathscr{L}_{A^{(1)},B^{(1)},C^{(1)},q^{(1)}}u(t,x)= \left(A(t,x)\Box+B(t,x)\partial_{t} +C(t,x)\cdot \nabla_{x}+q(t,x)\right)u_{2}(t,x), & \text { in } Q,\\
\partial_{t}^{l}u(0,x)=0, \hspace{5mm} l=0,1,2,3 &\text { on }\Omega,\\
\left( u|_{\Sigma}, (\Box u)|_{\Sigma}\right)= \left ( 0, 0\right) &\text { on } \Sigma,
\end{cases}
\end{aligned} 
\end{equation}
where
\begin{equation*}
    \begin{aligned}
        \begin{cases}
            A(t,x):=(A^{(2)}(t,x)-A^{(1)}(t,x)),\quad B(t,x):=(B^{(2)}(t,x)-B^{(1)}(t,x)),\\
            C(t,x):= (C^{(2)}(t,x)-C^{(1)}(t,x)) ,\quad q(t,x):=(q^{(2)}(t,x)-q^{(1)}(t,x)).
        \end{cases}
    \end{aligned}
\end{equation*}
Let $v\in H^{2}(Q)$ solves $\mathscr{L}_{A^{(1)}, B^{(1)},C^{(1)},q^{(1)}}^{\ast}v(t,x)=0$, in $Q$ be of the form \eqref{Exponential decaying solution form}.
Now, by using the assumption \eqref{IO-operator} of Theorem \ref{Main-Result}, we have 
\begin{equation}\label{Boundy_data}
\partial_{\nu}u|_{G}=\partial_{\nu}(\Box)u|_{G}=0, \quad  \partial_{t}^lu(T,x)=0\quad \text {for } l=0,1,2.
\end{equation}
Using equations \eqref{Equation for u} and \eqref{Boundy_data}, we calculate
\[
\int_{Q}  \mathscr{L}_{A^{(1)},B^{(1)},C^{(1)},q^{(1)}}u(t,x)\, \overline{v(t,x)}\, \D t\D x - \int_{Q} u(t,x)\, \overline{\mathscr{L}^{*}_{A^{(1)},B^{(1)},C^{(1)},q^{(1)}}v(t,x)}\, \D t\D x = \Bc,
\]
where
\begin{equation*}
\begin{aligned}
    \Bc &= \int_{\Omega}\partial_{t}(\Box u(T,x))\,\overline{v(T,x)}\, \D x 
    +\int_{\Sigma\setminus G}(\partial_{\nu}(\Box u(t,x)))\, \overline{v(t,x)}\,\D S_{x}\, \D t \\
    &+\int_{\Sigma\setminus G}\partial_{\nu}u(t,x) \, \overline{\Delta v(t,x)}\, \D S_{x}\, \D t.
\end{aligned}
\end{equation*}
Using \eqref{Equation for u}, we thus obtain
\begin{equation}\label{Final Integral Identity}
\begin{aligned}      \int_{Q}\left(A(t,x)\Box+B(t,x)\partial_{t} +C(t,x)\cdot \nabla_{x}+q(t,x)\right)u_{2}(t,x)\overline{v(t,x)}\D x \D t = \Bc.
\end{aligned}
\end{equation}
We estimate the terms in $\mathcal{B}$ using the following lemma.

\begin{lemma}\label{Boundary-Vanishes}
    Let $u_{i}$ for $i=1,2$ solutions to \eqref{Equation for ui} with $u_{2}$ of the form \eqref{Exponential growing solution form}.
 Let $u=u_{1}-u_{2},$ and $v$ be of the form \eqref{Exponential decaying solution form}.
Then 
    \begin{equation}\label{1st boundary term}
\lim_{h\to 0}\int_{\Omega}h\partial_{t}(\Box u(T,x))\,\overline{v(T,x)}\,\D{x}=0,
    \end{equation}
   \begin{equation}\label{2nd boundary term}
       \lim_{h\to 0}\int_{\Sigma\setminus G}h\partial_{\nu}(\Box u(t,x)\,\overline{v(t,x)}\,\D{S_{x}}\D{t}=0,
   \end{equation}
   \begin{equation}\label{3th boundary term}
       \lim_{h\to 0}\int_{\Sigma\setminus G}h\partial_{\nu}u(t,x)\, \overline{\Delta v(t,x)}\,\D{S_{x}}\D{t}=0,
   \end{equation}
 holds for all
\begin{equation}\label{nbhd}
  \omega\in \mathcal{N}_{\ve}(\omega_{0})=\Bigl\{y\in \mathbb{S}^{n-1}: |y-\omega_{0}|\leq \ve  \quad \text{ for some }  \ve>0 \Bigr\}.
\end{equation}
\end{lemma}

\begin{proof}
First, we prove \eqref{1st boundary term}. Using \eqref{decay_est} and the expression of $v$ from \eqref{Exponential decaying solution form} and the Cauchy-Schwartz inequality, we get
\begin{equation}\label{est_1}
 \begin{aligned}	\left|h\int\limits_{\Omega}\partial_{t}(\Box u(T,x))\overline{v(T,x)}\D x\right|
 &\leq \int\limits_{\Omega}h\left|\partial_{t}(\Box u(T,x))e^{-\frac{\vp(T,x)}{h}}\overline{\left(b_{d_{0}}(T,x)+hb_{d_{1}}(T,x)+ r_{d}(T,x;h)\right)}\right|\D x \\
 	&\leq C\,\left(\int\limits_{\Omega}h^{2}\left|\partial_{t}(\Box u(T,x))e^{-\frac{\vp(T,x)}{h}}\right|^{2}\D x\right)^{\frac{1}{2}},
 	\end{aligned}
  \end{equation}
  for some $C$ independent of $h$.
  
Now, using the boundary Carleman estimate \eqref{Boundary-Carleman-estimate}, we get 
\begin{equation}\label{Car_est}
    \begin{aligned}
    h^{7}\int\limits_{\Omega}\left|\partial_{t}(\Box u(T,x))e^{-\frac{\vp(T,x)}{h}}\right|^{2}\D x
    \leq&\, C\lVert h^{4}e^\frac{-\varphi}{h}\mathscr{L}_{A^{(1)},B^{(1)},C^{(1)},q^{(1)}}u\rVert_{L^{2}(Q)}^{2}\\
    \leq& \,C\lVert h^{4}e^{\frac{-\varphi}{h}}\left(A\Box +B\partial_{t}+C\cdot \nabla_{x}+q\right)u_{2}\rVert_{L^{2}(Q)}^{2}.
\end{aligned}
\end{equation}
Using expression for $u_2$ from \eqref{Exponential growing solution form} in \eqref{Car_est} to get 
 \begin{align}\label{est_2}
 \int\limits_{\Omega}\left|\partial_{t}(\Box u(T,x))e^{-\frac{\vp(T,x)}{h}}\right|^{2}\D x \leq \frac{C}{h}.
 \end{align}
Therefore from \eqref{est_1}, we obtain 
 \[
 \lim_{h\to 0}\int_{\Omega}h\partial_{t}(\Box u(T,x))\overline{v(T,x)}\D{x}=0.
 \]
Next, we prove \eqref{2nd boundary term}. Using \eqref{decay_est} and the expression of $v$ from \eqref{Exponential decaying solution form}, in the left-hand side of \eqref{2nd boundary term}, we have

\begin{equation}
\begin{aligned}\label{Estimate-2}	
\left|  h\int_{\Sigma\setminus G}(\partial_{\nu}(\Box u(t,x))\overline{v(t,x)}\D{S_{x}}\D{t}\right| &\leq \int\limits_{\Sigma\setminus G}h \left|(\partial_{\nu}(\Box u(t,x))e^{-\frac{\vp(t,x)}{h}}\overline{\left(b_{d_{0}}+hb_{d_{1}}+ r_{d}\right)}\right| \\
    & \leq C\,\left( \int\limits_{\Sigma\setminus G}h^{2}\left|(\PD_{\nu}(\Box u(t,x))e^{-\frac{\vp(t,x)}{h}}\right|^{2}\right)^{\frac{1}{2}}.
 \end{aligned}
 \end{equation}
 For $\ve>0$, define 
 \begin{align*}
     \partial\Omega_{+,\ve,\omega}=\{x\in \partial \Omega: \nu(x) \cdot \omega > \ve\} \text { and } \Sigma_{+,\ve,\omega}=(0,T)\times \partial \Omega_{+,\ve,\omega}.
 \end{align*}
Then from the definition of $G$ we see that there exists $\ve>0$ such that for all $\omega\in \mathcal{N}_{\ve}(\omega_{0})$, we have $(\Sigma\setminus G)\subset \Sigma_{+,\ve,\omega}$. Using this, we obtain
\begin{equation*}
    \begin{aligned}
        \int_{\Sigma\setminus G}\left|\PD_{\nu}(\Box u(t,x))e^{-\frac{\vp(t,x)}{h}}\right|^{2}\D{S_{x}}\D{t}&\leq \int_{\Sigma_{+,\varepsilon,\o}}\left|\PD_{\nu}(\Box u(t,x))e^{-\frac{\vp(t,x)}{h}}\right|^{2}\D{S_{x}}\D{t}\\
&\leq \frac{1}{\ve}\int_{\Sigma_{+,\ve,\omega}}(\PD_{\nu}\vp)\left|\PD_{\nu}(\Box u(t,x))e^{-\frac{\vp(t,x)}{h}}\right|^{2}\D{S_{x}}\D{t},
    \end{aligned}
\end{equation*}
for all $\omega\in \mathcal{N}_{\ve}(\omega_{0})$.
Now using the boundary Carleman estimate \eqref{Boundary-Carleman-estimate}, we get
\begin{align*}
h^{3}\int_{\Sigma_{+,\ve,\omega}}\PD_{\nu}\vp
\left|\PD_{\nu}(h^{2}\Box u(t,x))e^{-\frac{\vp(t,x)}{h}}\right|^{2}\D{S_{x}}\D{t}&\leq C\lVert e^{-\frac{\vp}{h}}h^{4}\mathscr{L}_{A^{(1)},B^{(1)},C^{(1)},q^{(1)}}u\lVert_{L^{2}(Q)}^{2}\\
&\leq \,C\lVert h^{4}e^{\frac{-\varphi}{h}}\left(A\Box +B\partial_{t}+C\cdot \nabla_{x}+q\right)u_{2}\rVert_{L^{2}(Q)}^{2}.
\end{align*}
Using the form of $u_2$ from \eqref{Exponential growing solution form} on the right side of the above inequality, we get
\[
\int_{\Sigma_{+,\ve,\omega}}(\PD_{\nu}\vp)
\left|\PD_{\nu}(\Box u(t,x))e^{-\frac{\vp(t,x)}{h}}\right|^{2}\D{S_{x}}\D{t} \leq \frac{C}{h}.
\]
Hence, from \eqref{Estimate-2}, we get 
\[
\lim_{h\to 0}\int_{\Sigma\setminus G}h\partial_{\nu}(\Box u(t,x)\overline{v(t,x)}\D{S_{x}}\D{t}=0.
\]
Following a similar analysis, we obtain \eqref{3th boundary term}. This completes the proof.
\end{proof}

\begin{remark}\label{Boundary_term_remark}
If we assume $A^{(1)}=A^{(2)}$, $B^{(1)}=B^{(2)}$ and $C^{(1)}=C^{(2)}$ in $Q$. Then, using the same techniques in Lemma \ref{Boundary-Vanishes}, one can show that
\begin{equation}
\begin{gathered}
\lim_{h\to 0}\int_{\Omega}\partial_{t}(\Box u(T,x))\,\overline{v(T,x)}\,\D{x}=0, \qquad 
\lim_{h\to 0}\int_{\Sigma\setminus G}\partial_{\nu}(\Box u(t,x)\,\overline{v(t,x)}\,\D{S_{x}}\D{t}=0,\\
\mbox { and } \quad 
\lim_{h\to 0}\int_{\Sigma\setminus G}\partial_{\nu}u(t,x)\, \overline{\Delta v(t,x)}\,\D{S_{x}}\D{t}=0,
\end{gathered}
\end{equation}
holds for all $\omega\in \mathcal{N}_{\ve}(\omega_{0})$.
\end{remark}

\subsection{Recovery of the coefficients}
In this section we prove Theorem ~\ref{Main-Result}, that is, $A^{(1)}=A^{(2)}$, $B^{(1)}=B^{(2)}$, $C^{(1)}=C^{(2)}$ and $q^{(1)}=q^{(2)}$ in $Q$ from the integral identity $\eqref{Final Integral Identity}$. Our approach is a generalization of \cite{Kian_Damping_2016,K.Vashisth}. 
We substitute the expression for $u_{2}$ and $v$ from \eqref{Exponential growing solution form} and \eqref{Exponential decaying solution form} into \eqref{Final Integral Identity}, to obtain
\begin{equation}\label{Integral-Identity}
\begin{aligned}
&\int_{Q}\Big \{ A(t,x)\left(\frac{2}{h}(\PD_{t}-\omega\cdot \nabla_{x})+ (\PD_{t}^{2}-\Delta_{x})\right)+B(t,x)\left(\frac{1}{h}+\partial_{t}\right)+C(t,x)\cdot\left(\frac{\omega}{h}+\nabla_{x}\right)+q(t,x)\Big\}\\
&\hspace{8.5cm} \times \left(a_{g_{0}}+ha_{g_{1}}+r_{g}\right)\overline{\left(b_{d_{0}}+hb_{d_{1}}+r_{d}\right)}\, \D{x}\D{t}=\mathcal{B}.
\end{aligned}
\end{equation}
Multiplying $h$ with \eqref{Integral-Identity} and using the bounds for $r_g$, $r_d$ from \eqref{decay_est}, \eqref{Exponential decaying solution form}, we get
\begin{equation*}
\begin{aligned}
&\int_{Q}\left( 2A(t,x)(\PD_{t}-\omega\cdot \nabla_{x})+B(t,x)+\omega\cdot C(t,x)\right) a_{g_{0}}\overline{b_{d_{0}}} \, \D{x}\D{t} +  
\int_{Q} \mathcal{O}(h) = h\mathcal{B}.
\end{aligned}
\end{equation*}
Further, applying Lemma \ref{Boundary-Vanishes} on $\mathcal{B}$ and taking $h\to 0$, we obtain
\begin{equation}\label{Identity-1}
\int_{Q}\Big \{2A(t,x)(\partial_{t}-\omega\cdot\nabla_{x})a_{g_{0}}\overline{b_{d_{0}}}
+\widetilde{B}(t,x)\cdot(1,-\omega)a_{g_{0}}\overline{b_{d_{0}}}\Big\}\, \D{x}\D{t}=0.
\end{equation}
Here $\widetilde{B}(t,x):=(B, -C)(t,x)$ and holds for all $\omega\in \mathcal{N}_{\ve}(\omega_{0})$. 
Now we choose $a_{g_{0}}=e^{-\mathrm{i}\xi\cdot(t,x)}$ and $b_{d_{0}}=1$, with $\xi\cdot(1,-\omega)=0$. Note that, with the above choices, $a_{g_{0}}$ and $b_{d_{0}}$ solves the transport equations \eqref{First-Transport-eqn} and \eqref{Transport for adjoint eqns} respectively. Since $\widetilde{B}^{(1)},\widetilde{B}^{(2)} \in W^{2,\infty}(Q)$ and $\widetilde{B}^{(1)}=\widetilde{B}^{(2)}$ on $\partial Q$, we extend $\widetilde{B}$ on $\R^{1+n}\setminus Q$ by zero and denote the extended function by the same notation. Using the above choices of $a_{g_0}$ and $b_{d_0}$, \eqref{Identity-1} now reads
\begin{equation}\label{Identity-2}
    \int_{\R^{1+n}}(1,-\omega)\cdot \widetilde{B}(t,x)e^{-\mathrm{i}\xi\cdot (t,x)}\D{x}\D{t}=0, \quad \text{where $\xi\cdot (1,-\omega)=0$ for all $\omega\in\mathcal{N}_{\ve}(\omega_{0})$}.
\end{equation}
Decomposing $\R^{1+n}=(1,-\omega)^{\perp}\oplus \R(1,-\omega)$ and using Fubini's theorem we see
\begin{align}\label{Iden_Fourier}
\int_{(1,-\omega)^{\perp}}e^{-\mathrm{i}\xi\cdot k}\left(\sqrt{2}\int_{\R}(1,-\omega)\cdot \widetilde{B}(k+\tau(1,-\omega))\D{\tau}\right)\D{k}=0, 
\end{align}
holds true for all $\omega\in\mathcal{N}_{\ve}(\omega_{0})$, where $\D k$ denotes the Lebesgue measure on $(1,-\omega)^{\perp}$.
Let 
\[
\phi(y) := \sqrt{2}\int_{\R}(1,-\omega)\cdot \widetilde{B}(y+\tau(1,-\omega))\,\D{\tau}
\in L^{1}((1,-\omega)^{\perp}).
\]
From \eqref{Iden_Fourier} and the injectivity of the Fourier transform, we get
\[ 0 = \phi(y)
= \sqrt{2}\int_{\R}(1,-\omega)\cdot \widetilde{B}(y+\tau(1,-\omega))\,\D{\tau}
\quad \mbox{for all } y \in (1,-\omega)^{\perp} \text { and } \omega \in \mathcal{N}_{\ve}(\omega_{0}).
\]
Now, we use a change of variables to get the above identity for all $(t,x)\in \R^{1+n}$. For any $(t,x)\in \R^{1+n}$ and we see 
\[
\left(\frac{t+x\cdot\omega}{2}, x+\frac{t-x\cdot\omega}{2}\omega\right) \in (1,-\omega)^{\perp}.
\]
Therefore,
\[
\int_{\R}(1,-\omega)\cdot \widetilde{B}\left(\frac{t+x\cdot\omega}{2}+\tau, x+\frac{(t-x\cdot\omega)\omega}{2}-\tau\omega\right)\D{\tau}=0,
\]
for all $\omega\in\mathcal{N}_{\ve}(\omega_{0})$ and for all $(t,x)\in (1,-\omega)^{\perp}.$
Using a change in variable
\[
 s=\frac{x\cdot\omega -t}{2}+\tau,
 \]
we obtain
\begin{equation}\label{light_ray_vector_field}
    \int_{\R}(1,-\omega)\cdot \widetilde{B}(t+s,x-s\omega)\D{s}=0, \quad \text { for all $(t,x)\in \R^{1+n}$ and $\omega\in\mathcal{N}_{\ve}(\omega_{0})$ }.
\end{equation}
Using the injectivity result of the light ray transform for vector fields from \cite[Section 6.2]{K.Vashisth}, we obtain
$\wt{B}(t,x)=\nabla_{t,x}\Phi(t,x)$ for some $\Phi\in W^{3,\infty}(Q)$ with $\Phi|_{\partial Q}=0$.

Substituting $\wt{B}(t,x)=\nabla_{t,x}\Phi(t,x)$ in \eqref{Identity-1} with the choices $a_{g_{0}}(t,x)=e^{-\mathrm{i}\xi\cdot (t,x)}$, where $(1,-\omega)\cdot \xi=0$ and $b_{d_{0}}(t,x) = (1,-\omega) \cdot(t,x)$ we perform integration by parts, and obtain
\[
\int_{Q}\Phi(t,x)(1,-\omega)\cdot\left(\nabla_{t,x}\right) e^{-\mathrm{i}\xi\cdot(t,x)}\D{x}\D{t}=0, \quad \text{for all $\xi\in (1,-\omega)^{\perp}$ and $\omega \in \mathcal{N}_{\ve}(\omega_{0})$}.
\]
Extending $\Phi$ by zero outside $Q$, we get
\[
\int_{\R^{1+n}}\Phi(t,x)e^{-\mathrm{i}\xi\cdot(t,x)}\, \D{x}\D{t}=0, \quad \text { for all $\xi\in (1,-\omega)^{\perp}$ and $\omega \in \mathcal{N}_{\ve}(\omega_{0})$}.
\]
Therefore $\hat{\Phi}(\xi)=0$ for all $\xi \in \R^{1+n}$ such that $\xi\cdot (1,-\omega)$.
The set of all $\xi$ such that $\xi\in (1,-\omega)^{\perp}$ for some $\omega\in \mathcal{N}_{\ve}(\omega_{0})$ forms an open cone when $n\geq 3$ (see \cite{Siamak_2018}).
Since $\Phi$ is a compactly supported, using Paley-Wiener theorem then it follows that $\Phi(t,x)=0$, and hence $\wt{B}(t,x)=0$. This proves $B^{(1)}(t,x)=B^{(2)}(t,x) $ and $C^{(1)}(t,x)=C^{(2)}(t,x)$ in $Q$.

Next, to prove $A^{(1)}(t,x)=A^{(2)}(t,x)$ in $Q$, we go back to \eqref{Identity-1} and substitute $\wt{B}(t,x)=0$ to get
\begin{align}\label{Identity-5}
\int_{Q}A(t,x)\left((\partial_{t}-\omega\cdot \nabla)a_{g_{0}}(t,x)\right)\overline{b_{d_{0}}(t,x)}\, \D{x}\D{t}=0, \quad \text{for all $\omega \in \mathcal{N}_{\ve}(\omega_{0})$}.
\end{align}  
Now, take $a_{g_{0}}(t,x)=e^{-\mathrm{i}\xi\cdot(t,x)}(1,-\omega)\cdot(t,x)$, $b_{d_{0}}(t,x)=1$ in \eqref{Identity-5}, and extend $A\in W^{3,\infty}(Q)$ by zero outside $Q$ to obtain
\[
\int_{\R^{1+n}}A(t,x)e^{-\mathrm{i}\xi\cdot (t,x)}\D{x}\D{t}=0, \quad \text{ for all } \xi\in (1,-\omega)^{\perp} \text { and } \omega \in \mathcal{N}_{\ve}(\omega_{0}).
\]
This implies that $\hat{A}(\xi)=0$ for all $\xi \in (1,-\omega)^{\perp}$ where $\omega\in\mathcal{N}_{\ve}(\omega_{0})$. Since $A$ is compactly supported, using Paley–Wiener theorem we conclude $A(t,x)=0$ in $Q$. This shows $A^{(1)}(t,x)=A^{(2)}(t,x)$ in $Q$.

Now it remains to show $q^{(1)}(t,x)=q^{(2)}(t,x)$ in $Q$. Putting $A^{(1)}=A^{(2)}$, $B^{(1)}=B^{(2)}$ and $C^{(1)}=C^{(2)}$ in the integral identity \eqref{Final Integral Identity} we get
\[
\begin{aligned}
&\int_{Q}q(t,x)u_{2}(t,x)v(t,x)\, \D{x}\D{t}\\
&=
\int_{\Omega}\partial_{t}(\Box u(T,x))\,\overline{v(T,x)}\,\D{x} + 
\int_{\Sigma\setminus G}\partial_{\nu}(\Box u(t,x)\,\overline{v(t,x)}\,\D{S_{x}}\D{t}
+ \int_{\Sigma\setminus G}\partial_{\nu}u(t,x)\, \overline{\Delta v(t,x)}\,\D{S_{x}}\D{t}.
\end{aligned}
\]
Here we take $h \to 0$ and use Remark \ref{Boundary_term_remark} to obtain
\begin{equation}\label{Identity-6}
\int_{Q}q(t,x)a_{g_{0}}(t,x)\overline{b_{d_{0}}(t,x)}\D{x}\D{t}=0, \text { for all $\omega \in \mathcal{N}_{\ve}(\omega_{0})$.}
\end{equation}
Choosing $a_{g_{0}}(t,x)=e^{-\mathrm{i}\xi\cdot(t,x)}$ and $b_{d_{0}}(t,x)=1$ satisfying \eqref{First-Transport-eqn} and \eqref{Transport for adjoint eqns}, from \eqref{Identity-6} we obtain
\[
\hat{q}(\xi)= \int_{\R^{1+n}}q(t,x)e^{-\mathrm{i}\xi\cdot(t,x)}\D{x}\D{t}=0, \quad \text { for all } \xi\in (1,-\omega)^{\perp} \text { with } \omega \in \mathcal{N}_{\ve}(\omega_{0}).
\]
Since $q\in L^{\infty}(Q)$, extended by zero outside $Q$, therefore, using Paley--Wiener theorem we have $q^{(1)}(t,x)=q^{(2)}(t,x)$ in $Q$. This completes the proof of Theorem \ref{Main-Result}.

\section*{Acknowledgment}
The authors thank the Department of Mathematics at the Indian Institute of Science Education and Research (IISER), Bhopal, for their invaluable support. This work was initiated when PK was a graduate student at IISER-Bhopal during 2024. He is currently supported by the Postdoctoral Scholarship at YMSC, Tsinghua University.
Finally, we thank the anonymous reviewers for suggestions.

\bibliographystyle{alpha}
\bibliography{reference.bib}

\end{document}